\documentclass[12pt,oneside]{article}
\usepackage{amsmath,amstext,amssymb,amsthm}
\usepackage{hyperref}

\theoremstyle{plain}
\newtheorem{theorem}{Theorem} 
\newtheorem{lemma}{Lemma}
\newtheorem{proposition}{Proposition}
\newtheorem{corollary}{Corollary}
 \theoremstyle{definition}

\newtheorem{example}{Example} \theoremstyle{remark}
\newtheorem*{remark*}{Remark} \newtheorem{remark}{Remark}

 \newcommand{\R}{\mathbb{R}}
\newcommand{\Rd}{{\R^{d}}}

\renewcommand{\leq}{\leqslant}
\renewcommand{\geq}{\geqslant}
\newcommand{\Z}{\int^{\infty}_{0}}

\title{On Harnack inequality and H\"{o}lder regularity for isotropic unimodal L\'{e}vy processes }

\author{Tomasz Grzywny\\
Institute of Mathematics and Computer Sciences,\\
  Wroc\l{}aw University of Technology\\
ul. Wybrzeze Wyspianskiego 27\\
50-370 Wroclaw, Poland\\
Tel.: +48-71-3203153\\
Fax: +48-71-3280751\\
tomasz.grzywny@pwr.wroc.pl
}

\begin{document}


\maketitle

\begin{abstract}
We prove the scale invariant Harnack inequality and regularity properties for harmonic functions with respect to an isotropic unimodal L\'{e}vy process with  the characteristic exponent $\psi$ satisfying some scaling condition. We derive sharp estimates of the potential measure and capacity of balls, and further, under the assumption that $\psi$ satisfies the lower scaling condition, sharp estimates of the potential kernel of the underlying process. This allows us to establish the Krylov-Safonov type estimate, which is the key ingredient in the approach of Bass and Levin, that we follow.
\end{abstract}

\section{Introduction}

Let $X_t$ be a L\'{e}vy process with the characteristic exponent
$$\psi(x)=\left<x,Ax\right>-i\left<x,\gamma\right>-\int_{\R^d}\left(e^{i\left<x,z\right>}-1-i\left<x,z\right>{\textbf{1}_{|z|<1}}\right)\nu(dz), \quad z\in\R^d,$$
where  $A$ is  a symmetric and non-negative definite matrix, $\nu$ is a L\'{e}vy measure, i.e. $\nu(\{0\})=0$, $\int_{\Rd}\left(1\wedge|z|^2\right)\nu(dz)<\infty$
 and $\gamma\in\R^d$. A generator of this process has the following form, for $f\in C^2_b(\R^d)$,
\begin{equation}\label{generator}
{\cal A} f(x)= \sum_{j,k}A_{jk}\partial^2_{jk} f(x)+\left<\gamma,\nabla f(x)\right>+\int \left(f(x+z)-f(x)
-{\bf 1}_{|z|<1}\left<z,\nabla f(x)\right>\right)\nu(dz).
\end{equation}

As usual we denote by $P^x$ and $E^x$ the probability measure $P(\cdot|X_0=x)$ and the corresponding expectation, respectively.
The {\it first exit time} of an (open)
   set  $D\subset {\Rd}$ and the {\it first hitting time} to $D$ (closed)
   by the process $X_t$ is defined by the formula
   $$
   \tau_{D}=\inf\{t> 0:\, X_t\notin D\},\,\qquad T_{D}=\inf\{t>0:\,X_t\in D\}.
   $$
A function $f:\R^d\rightarrow [0,\infty)$ is said that is harmonic with respect to $X_t$ in an open set $D$ if for any bounded open set $B$ such that $\bar{B}\subset D$
$$f(x)=E^xf(X_{\tau_B}), \qquad x\in B.$$
The {\it scale invariant Harnack inequality} holds for a process $X_t$ if for any $R>0$ there exists a constant $C=C(R)$ such that for any function non-negative on $\Rd$ and  harmonic in a ball $B(0,r)$, $r\leq R$, $$\sup_{x\in B(0,r/2)}h(x)\leq C\inf_{x\in B(0,r/2)}h(x).$$
We say that the {\it global} scale invariant Harnack inequality holds if constant in the above inequality does not depend on $R$.

A  measure $m(dx)$ is isotropic unimodal if there exists a non-increasing function $m_0:(0,\infty)\rightarrow [0,\infty)$ such that $m(dx)=m_0(|x|)dx$, for $x\neq 0$. A process is isotropic unimodal if a transition probability $P_t(dx)$ is isotropic unimodal, for all $t>0$.

Important class of isotropic unimodal L\'{e}vy processes are subordinate Brownian motions.

Let $f$ be a positive function on $\R^d\setminus\{0\}$.  We say that
$f$ satisfies the {\it weak lower scaling condition} WLSC($\beta, \theta, C$), if $\beta>0$, $\theta\geq 0$, $C>0$, and
$$
 f(\lambda x)\ge
{C}\lambda^{\beta} f(x),\quad \mbox{for}\quad \lambda\ge 1, \quad |x|\geq\theta.
$$
If $f$ satisfies WLSC($\beta,0,C$), then we say
that $f$ satisfies the {\it global} weak lower scaling condition.

The main purpose of this paper is to prove the scale invariant Harnack inequality and regularity properties for harmonic functions with respect to an isotropic unimodal L\'{e}vy process with  the characteristic exponent satisfying the weak lower scaling condition.  Our main technical results are sharp estimates of the potential measure and capacity of balls, and further, under the assumption  that $\psi$ satisfies the weak lower scaling condition, sharp estimates of the potential kernel of the underlying process. This allows us to establish the Krylov-Safonov type estimate (see Proposition \ref{PropKrylovSafonov}), which says that there are $c$ and $\lambda<1$ such that for a closed set $A\subset B(0,\lambda r)$,
$$P^x(T_A<\tau_{B(0,r)})\geq c\frac{|A|}{|B(0,r)|},\qquad x\in B(0,\lambda r).$$
This estimate is the key ingredient of the proofs of the Harnack inequality and local H\"{o}lder continuity of harmonic functions in the approach of Bass and Levin (\cite{BL}) that we follow.

Our main contribution is the fact that we assume only a mild  condition for the characteristic exponent but we do not use in our proofs any properties  of the L\'{e}vy measure except it is isotropic and unimodal.
 Usually in the existing literature on the Harnack inequality for L\'evy processes  the assumptions are given in terms of the behaviour of the L\'evy measure (see \cite{SongVondracek}, Section 3) or the initial step relies on describing its behaviour (\cite{KSV2}).
  Our result seems to   be  important for application to subordinate Brownian motions. There are examples when  the characteristic exponent is known, while estimates for the L\'{e}vy measure are not.
 We should  also notice that  our approach allows to deal with  isotropic unimodal processes with the L\'{e}vy-Khinchine exponent behaving at infinity almost like the exponent for a Brownian motion, which to our best knowledge  were not treated in the literature, except a few particular cases. Namely, we can take $\psi(x)= |x|^2 l(|x|)$, where $l$ is slowly varying and goes to $0$ at infinity.   An example of such a process is for instance a process with density of its L\'{e}vy measure equal to $|x|^{-d-2}\log^{-2}(2+|x|^{-1})$. Moreover, our result  allows to  extend the scale invariant Harnack  inequality to its global version  for many processes.   For instance we get the global scale invariant Harnack inequality for $\alpha$-stable relativistic processes.

The main results of this paper are  following two theorems. The first one is the scale invariant Harnack inequality.
\begin{theorem}\label{Harnack}Let $d\geq 3$. Suppose that $X_t$ is isotropic and unimodal.
If $\psi$ satisfies WLSC$(\beta,\theta,C)$, then the scale invariant Harnack inequality holds. Moreover, if  $\psi$ satisfies the global  weak lower scaling condition, then the global scale invariant Harnack inequality holds.
\end{theorem}
The next theorem deals with  regularity of harmonic functions
 \begin{theorem}\label{Holder}Suppose that $X_t$ is isotropic and unimodal. Let $d\geq 3$ and $\psi$ satisfy WLSC$(\beta,\theta,C)$.
For any $R>0$  there exist  constants $c=c(R)$ and $\delta>0$ such that, for any $0<r\leq R$, and any bounded  function $h$, which is harmonic in $B(0,r)$,
  $$|h(x)-h(y)|\leq c||h||_{\infty}\left(\frac{|x-y|}{r}\right)^{\delta},\qquad x,y\in B(0,r/2).$$
\end{theorem}
\begin{remark}
The assumption $d\geq 3$ in the two theorems above can be removed in the case of  subordinate Brownian motions (see Theorem \ref{Harnack_SBM}). For a general isotropic unimodal L\'{e}vy process  $X_t$ the  assumption $d\geq 3$ assures not only  that $X_t$ is transient but  the function $r\to r^{d-\varepsilon} \psi_0(1/r)$, for $\varepsilon\in(0,1)$ is almost increasing, where  $\psi_0$ is the radial profile of $\psi$.  The last property  is necessary in our approach (see proofs of Lemma \ref{Laplace_potential_LP}, Proposition \ref{GreenLowerLP}). At the end of Section 4 the  case $d=1,\,2$ is discussed in more detail.
\end{remark}

 Recently there has been a lot of research concerning  non-local operators. For instance, the paper \cite{ChenKimKumagai} established the scale-invariant finite range parabolic Harnack inequality for a class of jump-type Markov processes on metric measure spaces. A class of special subordinate Brownian motions have been studied in \cite{KM}, where bounds for the densities of L\'{e}vy measure and potential measure and Harnack inequalities were established. Harnack inequalities and regularity estimates for harmonic function with respect to diffusion with jumps are proved in \cite{Foondun}.
Related work on discontinuous processes include
\cite{SongVondracek}, \cite{SiSV}, \cite{Mimica},  \cite{ChenKumagai} and \cite{ChenKumagai2}.
Therefore it is pertinent to comment on the differences
between our results and those of some related papers.
For the sake of comparison we present them  in the context of L\'{e}vy processes, however
 most of them are in a more general setting of Markov processes.
\begin{itemize}
\item One of the main assumptions in \cite{ChenKimKumagai} is that the density of the L\'{e}vy measure is comparable to $\frac{1}{f(|x|)|x|^{d}}$ on $B(0,1)$, where $f$ is a strictly increasing continuous function
and satisfies the following conditions. There exist $0<\beta_1\leq \beta_2$, and a constant $c$ such that
$$c^{-1}\left(\frac{r_2}{r_1}\right)^{\beta_1}\leq\frac{f(r_2)}{f(r_1)}\leq c\left(\frac{r_2}{r_1}\right)^{\beta_2},\qquad 0<r_1<r_2<\infty,$$
$$\int^r_0\frac{s}{f(s)}ds\leq c \frac{r^2}{f(r)}, \qquad r>0.$$

  One can easily check that the  lower scaling condition for $f$ implies  the weak lower scaling condition for the characteristic exponent, hence our assumption is much weaker than that from \cite{ChenKimKumagai}.  In \cite{ChenKumagai}, under the assumption that  the above  estimate for the density of L\'{e}vy measure holds on the whole space, the authors obtained the global parabolic Harnack inequality. In our context of isotropic unimodal L\'{e}vy processes the global lower scaling for $f$ is sufficient to get the global Harnack inequality (see Example 2 in subsection 3.4).
\item In \cite{SongVondracek}  the following  Krylov-Safonov type estimate (Lemma 3.4) was derived
$$P^x(T_A<\tau_{B(0,r)})\geq c \frac{\nu(4r)}{\int(1\wedge|z|^2/r^2)\nu(z)dz}|A|.$$ Such an estimate is sufficient in the proof of the Harnack inequality only if a density of L\'{e}vy measure satisfies similar conditions as in \cite{ChenKimKumagai}. However,  it will not work for $\nu(x)=\frac{1}{|x|^{d+2}\ln^2(2+|x|^{-1})}$ since applying it one  obtains $P^x(T_A<\tau_B)\geq \frac{c}{\ln r^{-1}}\frac{|A|}{|B_r|}$, for $r\leq 1/2$. Hence if $r$ goes to $0$ the term $\frac{c}{\ln r^{-1}}$ vanishes, which makes the above bound useless for the proof of the scale invariant Harnack inequality.
\item In \cite{KM} it was considered a class of special subordinate Brownian motions such that a subordinator has a non-increasing density of the L\'{e}vy measure. Moreover, there was some scaling assumption on the Laplace exponent of subordinator  in terms of its derivative.

In the present paper  the weak lower scaling condition for the Laplace exponent of the subordinator  is sufficient to obtain the Harnack inequality and we do not need to assume anything else about the L\'evy measure of the subordinator. This does not mean that our result covers all  the results of  \cite{KM}. Their proof is not based on the Krylov-Safonov type estimate and it works for a large class of  slowly varying Laplace exponents, while our approach does not cover that case. This is due to the fact that the Krylov-Safonov type estimate does not need to hold for the subordinate Brownian motions driven by  subordinators with  slowly varying Laplace exponents. On the other hand
 our results improve the results from \cite{Mimica}, where it was studied only a particular case of subordinate Brownian motion  with $\psi(x)=\frac{|x|^2}{\ln(1+|x|^2)}-1$.
\item Since we do not exclude a case when a Gaussian part is non-zero we  mention the paper \cite{ChenKumagai2}, where   diffusions with jumps were considered. In this paper the density of the L\'evy measure is assumed to be bounded from above $\nu(x)\leq c|x|^{-d-\alpha}$, for $|x|\leq 1$, where $\alpha\in(0,2)$. Hence the result can not be applied to the process with $A=Id$ and $\nu(x)=\frac{1}{|x|^{d+2}\ln^2(2+|x|^{-1})}$.  Notice that for any L\'evy process with a non-trivial Gaussian part ($\mathrm{rank} A=d$) the  WLSC property holds for the characteristic exponent.
\item In \cite{Foondun} it is assumed that, for any $r<1$, there exist constants $c$ and $\alpha$ such that $\nu(x-z)\leq cr^{-\alpha}\nu(y-z)$, for $|x-y|<r$ and $|x-z|>r$. Therefore for instance this result does not cover the case $A=Id$ and $\nu(x)=\frac{1}{|x|^{d+2}\ln^2(1+|x|^{-1})}e^{-|x|^2}$,
 for which we even have the global Harnack inequality, due to Theorem \ref{Harnack}, since $\psi(x)\approx |x|^2$.
\end{itemize}

The paper is organized as follows. In Section 2, we give some preliminary results for general L\'{e}vy processes. Section 3 is devoted to prove estimates of Green function and the main results. Moreover, several examples  are presented to which   our approach applies.  In Section 4 some conditions are stated that are sufficient to prove the scale invariant Harnack inequality for L\'evy processes not necessarily isotropic and unimodal.

\section{Preliminaries}

In this section we introduce  notation and prove some auxiliary results for general L\'{e}vy processes.
We denote incomplete Gamma functions by
$$\gamma(\delta,t)=\int^t_0e^{-u}u^{\delta-1}du,\qquad \Gamma(\delta,t)=\int^\infty_te^{-u}u^{\delta-1}du,\quad \delta,\,t>0.$$
Let $B(x,r)$ denote a ball of center  $x$ and radius $r>0$ and let $B_r=B(0,r)$. By $\mathcal{L}$ we denote the Laplace transform, that means, for a measure $\mu$ on $[0,\infty)$,
$$\mathcal{L}\mu(\lambda)=\Z e^{-\lambda s}\mu(ds),\quad \lambda\geq 0.$$
For two non-negative functions $f$ and $g$ we write $f(x) \approx g(x)$ if  there is a positive number $C$ (i.e. a constant) such
that $C^{-1}f(x) \leq g(x) \leq C f(x)$. This $C$ is called a comparability constant. We write $C = C(a,\ldots, z)$ to emphasize that $C$  depends only on $a,\ldots, z$.  An integral $\int^b_a \ldots $ we understand as $\int_{[a,b)} \ldots $.

Our primary object is a potential measure $G$, which is well  defined for a  transient process, by the following formula
$$G(x,A)=\int^\infty_0P^x(X_t\in A)dt=E^x\int^\infty_0\textbf{1}_A(X_t)dt,$$
where $A$ is a Borel subset of $\R^d$. In what follows we always consider  Borel subsets of $\R^d$ without further mention.
Let $G(A)=G(0,A)$. Notice that $G(x,A)=G(A-x)$.  By a slight abuse of notation we also use   $G$ to  denote  the density of the absolutely continuous (with respect to the Lebesque measure)  part of the potential measure and then we call  $G(x,y)=G(y-x)$  a potential kernel.

The fundamental object of the potential theory is the {\it killed process} $X_t^D$
  when exiting the set $D$. It is defined in terms of sample paths up to time $\tau_D$.
  More precisely, we have the following formula:
  $$
  E^x f(X_t^D) =  E^x[t<\tau_D; f(X_t)]\,,\quad t>0\,.
  $$

 The potential measure of the process $X_t^D$ is called the
  {\it Green measure} and is denoted by $G_D$. That is
  $$G_D(x,A)=E^x\int^{\tau_D}_0\textbf{1}_A(X_t)dt.$$
  The corresponding kernel will be called the {\it Green function} of the set $D$ and denoted  $G_D(x,y)$. If the potential measure is  absolutely continuous, then  we have
 \begin{equation}\label{Hunt1}G_{D}(x,y)=G(y-x)-E^x\, G (X_{\tau_{D}}-y). \end{equation}

  Another important  object in the potential theory of $X_t$ is the
  {\it harmonic measure}  of the
  set $D$. It is defined by the formula:
  $$
  P_D(x,A)=
  E^x[\tau_D<\infty; {\bf{1}}_A(X_{\tau_D})].
  $$
  The density kernel (with respect to the Lebesgue measure) of  the measure $P_D(x,A)$ (if it exists) is called the
  {\it Poisson kernel} of the set $D$.
  The relationship between the Green function of $D$ and the harmonic measure is provided by the Ikeda-Watanabe formula \cite{IW},
\begin{equation} \label{IkedaWatanabe}
  P_D(x,A)=  \int_D \nu(A-y){G_D(x,dy)} , \quad A\subset (\bar{D})^c.
  \end{equation}

Important examples of isotropic unimodal L\'{e}vy processes are subordinate Brownian motions and some of  our results are
restricted to this class of processes. By $T_t$ we denote a subordinator i.e. a non-decreasing L\'{e}vy process starting from $0$. The Laplace transform of  $T_t$ is of the form
$$Ee^{-\lambda T_t}=e^{-t\phi(\lambda)},\qquad \lambda\geq 0,$$
where $\phi$ is called the Laplace exponent of $T$. $\phi$ is a Bernstein function and has the following representation:
$$\phi(\lambda)=b \lambda+\int_{(0,\infty)}(1-e^{-\lambda u})\mu(du),$$
where $b \ge 0$ and $\mu$ is a L\'{e}vy measure on $(0, \infty)$ such that $\int (1\wedge u) \mu(d u) < \infty$.

The potential measure of the subordinator $T$ is denoted by $U$.
Its Laplace transform is equal to
\begin{equation}\label{PotMeasureSubLaplace}\mathcal{L}U(\lambda)=\Z e^{-\lambda s} U(ds)=\frac{1}{\phi(\lambda)}.\end{equation}

We say that a Bernstein function $\phi$ is special if there exists a decreasing positive density $u$ on $(0,\infty)$ of a  measure $U|_{(0,\infty)}$. For a different characterization of special Bernstein functions see e.g. \cite{SSV}.

Let $B_t$ be a Brownian motion in $\Rd$ with the characteristic function of the form
$$Ee^{i\xi B_t}=e^{-t|x|^2},\quad  x\in \R^d.$$
By $g_t(x)$ we denote the  transition density of $B_t$.
Assume that $B_t$ and $T_t$ are stochastically independent. Then the process $X_t=B_{T_t}$ defines a subordinate Brownian motion. It is clear that
the characteristic function of $X_t$ takes the form
$$Ee^{i\xi X_t}=e^{-t\phi(|x|^2)}, \quad  x\in \R^d.$$
The  L\'{e}vy measure of the process $X_t$ is given by the following formula for its density
 $$\nu(x)dx=\int^\infty_0 g_u(x)\mu(du)dx,$$
while its  potential measure  is equal to
\begin{equation}\label{PotentialDensitySub}G(A)=\frac{1}{\lim_{\lambda\to \infty}\phi(\lambda)}\textbf{1}_{\{0\}}(A)+\int_A\Z g_s(y)U(ds)dy.\end{equation}
A subordinator which has a special Laplace exponent $\phi$ is called a special subordinator and the corresponding subordinate Brownian motion is called a special subordinate Brownian motion.

For a function $f: \R^d\to \mathbb{C}$ we define $f^*(u)=\sup_{|x|\leq u}\Re f(x)$. The following lemma will play an important role in the sequel.
\begin{lemma}\label{fStarScalling}
Let  $f: \R^d\to \mathbb{C}$ be a negative definite function, then
$$\frac{1}{2}\frac{s^2}{s^2+1}f^*(r)\leq f^*(s r)\leq 2(1+s^2)f^*(r),\qquad s,r>0.$$
\end{lemma}
\begin{proof}Since $f$ is negative definite, $\Re f(x)$ and $f_r(x)=f(rx)$ are negative definite functions as well. The upper bound we get  e.g.  by using \cite{Schilling1998}, (1.4) for $\Re f_r$. If $s\geq 1$, then we get the lower bound by monotonicity of $f^*$. For $s<1$,  by the upper bound
$$f^*(r)=f^*(rss^{-1})\leq 2(1+s^{-2})f^*(rs),$$
which completes the proof.
\end{proof}

\begin{lemma}\label{star}  Let  $f: \R^d\to [0,\infty)$ and $\tilde{f}(u)=\sup_{|x|= u} f(x)$. Suppose that $\tilde{f}$ is positive on $(0,\infty)$ and $f(0)=0$. If $f$ satisfies WLSC$(\beta, \theta, C)$, then  WLSC$\left(\beta, \theta, C^2\frac{\tilde{f}(\theta)}{f^*(\theta)}\right)$ holds for $f^*$, where $\frac{\tilde{f}(0)}{f^*(0)}=1$.
\end{lemma}

\begin{proof} We  assume that $\theta>0$, since the proof in the case $\theta=0$ is similar.
Note that,  WLSC$(\beta, \theta, C)$ holds for $\tilde{f}$. Hence,
\begin{equation}\label{tilde}\tilde{f}(u)\leq C^{-1}\tilde{f}(s),\qquad \theta \leq u\leq s.\end{equation}
Let  $u\geq\theta$. Since $\tilde{f}$ is positive on $(0,\infty)$,
$$
f^*(u)=\max\{f^*(\theta),\sup_{\theta\leq |x|\leq u} f(x)\}\leq  \frac{f^*(\theta)}{\tilde{f}(\theta)}\sup_{\theta\leq |x|\leq u} f(x)\leq C^{-1}\frac{f^*(\theta)}{\tilde{f}(\theta)}\tilde{f}(u).
$$
Hence, for $u\geq \theta$ and $\lambda>1$, applying again WLSC$(\beta, \theta, C)$ for $\tilde{f}$ we arrive at
$$f^*(\lambda u)\geq \tilde{f}(\lambda u) \geq C \lambda^{\beta}\tilde{f}(u)\geq C^2 \frac{\tilde{f}(\theta)}{f^*(\theta)}\,\lambda^{\beta}f^*( u).$$
\end{proof}

Until the end of this section we assume that $X_t$ is a L\'evy process characterized by a triplet $(A, \nu, \gamma)$.

In the proof of following lemma we follow closely the ideas of \cite{SongWu}, where authors proved similar result for isotropic stable processes.
\begin{lemma}\label{Harm1}Let $D\subset B_r$ and $x\in D\cap B_{r/2}$.
Then there is a constant $C=C(d)$ such that
$$P^x\left(|X_{\tau_{D}}|\ge r\right)=P_{D}(x,B_r^c)\le C h(r) E^x\tau_{D},$$
where  $$h(r)= ||A||r^{-2}+{\left|\gamma+\int z \left(\textbf{1}_{|z|<r}-\textbf{1}_{|z|<1}\right)\nu(dz)\right|}r^{-1}+\int_{\Rd} \min(1, |z|^2r^{-2})\ \nu(dz).$$
\end{lemma}
\begin{proof} For $f\in C^2_b(\Rd)$, by Dynkin formula we have
\begin{eqnarray}\label{gen}
	G_D({\cal A}f)(x)=  E^xf(X_{\tau_{D}})-f(x),\quad x\in \Rd.
\end{eqnarray}

There is a non-decreasing function $g\in C^2_b([0,\infty))$  such that $g(s)=0$, for $0\le s\le 1/2$, $g(s)=1$, for $s\ge 1$.
Let $c_1=\sup_{s}|g''(s)|$, then $\sup_{s}\frac{g'(s)}{s}\leq \frac{c_1}{2}$. We put $f(y)=g(|y|)$ and $f_r(y)=f(yr^{-1})$.
Recall that $\mathrm{Tr}(A)\leq d||A||$. Hence
\begin{equation}\label{DiffPartBound}\sum_{j,k}A_{jk}\partial^2_{jk} f(y)=\left(g''(|y|)-\frac{g'(|y|)}{|y|}\right)\frac{\left<y,Ay\right>}{|y|^2}+\frac{g'(|y|)}{|y|}\mathrm{Tr}(A)\leq c_1\left(1+\frac{d}{2}\right)||A||.\end{equation}
Since $g'(s)=0$, for $s\geq 1 $,
 $$\left<z,\nabla f(y)\right>=g'(|y|)\frac{\left<z,y\right>}{|y|}\leq \frac{c_1}{2}|z|.$$
By (\ref{DiffPartBound}), for $|z|<r$,
$$F_r(y,z)=f_r(y+z)-f_r(y)
-{\bf 1}_{|z|<r}\left<z,\nabla f_r(y)\right>\leq\frac{|z|^2}{2}\sup_y \sum_{j,k}\partial^2_{jk} f_r(y) \leq c_1\frac{d+2 }{4} \frac{|z|^2}{r^2}.$$
And, for $|z|\geq r$, $F_r(x,y)\leq 1$. By (\ref{generator}) we have
\begin{eqnarray}
{\cal A} f_r(y)&=& \int_{\R^d} F_r(y,z)\nu(dz)+ \left<\gamma+\int_{\R^d}\left({\bf 1}_{|z|<r}-{\bf 1}_{|z|<1}\right)z\nu(dz),\nabla f_r(y)\right>\nonumber\\&&+\sum_{j,k}A_{jk}\partial^2_{jk} f_r(y)\nonumber\\
&\le& c_1\frac{d+2}{4}\int_{|z|<r}\left(1\wedge \frac{|z|^2}{r^2}\right)\nu(dz) + c_1\frac{d+2}{2}\frac{||A||}{r^2}\nonumber\\&&+\frac{c_1}{2}\frac{\left|\gamma+\int_{\R^d}\left({\bf 1}_{|z|<r}-{\bf 1}_{|z|<1}\right)z\nu(dz)\right|}{r}.\label{gen2}
\end{eqnarray}
Applying
(\ref{gen}) to $f_r(y)$, we get
\begin{eqnarray}\label{gen1}
	G_D({\cal A}f_r)(x)=  E^x f_r(X_{\tau_{D}}), \quad |x|\le r/2.
\end{eqnarray}
Since $P^x\left(|X_{\tau_{D}}|\ge r\right)\le E^x f_r(X_{\tau_{D}})$ and $G_D \textbf 1(x)=E^x\tau_D$, the estimates (\ref{gen2}) and (\ref{gen1}) provide the conclusion.
\end{proof}

\begin{lemma}\label{PsiStarApprox}
For any $r\geq 0$,
$$\frac{1}{8(1+2d)}\left(||A||r^2+ \int_\Rd\left(1\wedge (r|z|)^2\right)\nu(dz)\right) \leq \psi^*(r)\leq 2 \left(||A||r^2+ \int_\Rd\left(1\wedge (r|z|)^2\right)\nu(dz)\right).$$
\end{lemma}
\begin{proof}
Let us observe that
$$\psi^*(r)\leq \left(\sup_{|z|\leq r}\left<z,Az\right>+\sup_{|z|\leq r}\int_\Rd(1-\cos\left<z,y\right>)\nu(dy)\right)\leq 2 \psi^*(r).$$
Since $\sup_{|z|\leq r}\left<z,Az\right>={||A||}{r^2}$  it remains to prove
\begin{equation}\label{hApprox1}
\frac{1}{4(1+2d)}\int_{\Rd} \min(1, (|z|r)^2)\ \nu(dz)\leq\sup_{|z|\leq r}\int_\Rd(1-\cos\left<z,y\right>)\nu(dy)\leq 2\int_{\Rd} \min(1, (|z|r)^2)\ \nu(dz).
\end{equation}
 Let $\tilde{\psi}(z)=\int(1-\cos\left<z,y\right>)\nu(dy)$.
 Notice that (see e.g. \cite{JakobSchilling}, (5.4)),$$\frac{|x|^2}{1+|x|^2}=\int_{\R^d}\left(1-\cos(\left<x,y\right>)\right)g(y)dy,$$
 where
 $$g(y)=\frac{1}{2}\Z (2\pi s)^{-d/2}e^{-\frac{|y|^2}{2s}}e^{-\frac{s}{2}}ds.$$
Hence, by the Fubini-Tonelli theorem
\begin{eqnarray*}\int_{\Rd} \min(1, (|z|r)^2)\ \nu(dz)&\leq& 2\int_{\Rd} \frac {(|z|r)^{2}}{1+(|z|r)^{2}} \nu(dz)\\
&=&2\int_{\Rd} \int_{\R^d}\left(1-\cos(\left<zr,y\right>)\right)g(y)dy \nu(dz)\\&=&2\int_{\R^d}\tilde{\psi}(yr)g(y)dy.
\end{eqnarray*}

Since $\tilde{\psi}$ is a negative definite function, by Lemma \ref{fStarScalling} we have
$$\int_{\Rd} \min(1, (|z|r)^2)\ \nu(dz)\leq 4 \sup_{|z|\leq r}\tilde{\psi}(z)\int_{\Rd}(1+|y|^2)g(y)dy=4(1+2d)\sup_{|z|\leq r}\tilde{\psi}(z).$$
Since $1-\cos u=2\sin^2\frac{u}{2}\leq 2\left(1\wedge |u|^2\right)$,
$$\int_{\Rd} \min(1, |z|^2|x|^2)\ \nu(dz)\geq \frac{1}{2}\tilde{\psi}(x).$$
Hence
$$\sup_{|x|\leq r}\tilde{\psi}(z)\leq 2\sup_{|x|\leq r} \int_{\Rd} \min(1, (|z||x|)^2)\ \nu(dz)=2\int_{\Rd} \min(1, (|z|r)^{2})\ \nu(dz),$$
which completes the  proof of  the inequality (\ref{hApprox1}).
\end{proof}

Since for symmetric processes $h(r)=\frac{||A||}{r^{2}}+ \int\left(1\wedge \frac{|z|^2}{r^2}\right)\nu(dz)$, we obtain the following corollary.
\begin{corollary}\label{hApproxLemma}Let $X_t$ be symmetric, then
\begin{equation}\label{hApprox}\frac{1}{2}\psi^*(r^{-1})\leq h(r)\leq 8(1+2d)\psi^*(r^{-1}) .\end{equation}
\end{corollary}
\begin{remark}\label{hApproxRem}
Instead of a direct calculation one can compare  Pruitt's result \cite{Pruit} and \cite{Schilling1998}, Remark 4.8 to obtain $\psi^*(r^{-1})\approx h(r)$ under the assumption that there exists a constant $c$ such that $|\Im \psi(x)|\leq c \Re \psi (x)$, $x\in \Rd$.
\end{remark}

For subordinate Brownian motions  easy calculations  improve (\ref{hApprox}). Notice that $\phi$ is increasing.
\begin{remark}Let $X_t$ be a subordinate Brownian motion, then for $r>0$,$$ \frac{1}{2}\phi(r^{-2})\leq h(r)\leq (1+2d)\phi(r^{-2}).$$
\end{remark}
%

\begin{corollary}\label{JumpProb}Let $X_t$ be  symmetric.
There exists a constant $C=C(d)$ such that, for $r>0$, $s\leq r/2$,
\begin{equation}\label{JumpProb1}P^x(|X_{\tau_{B(0,s)}}|\geq r)\leq  C\frac{\psi^*(r^{-1})}{\psi^*(s^{-1})},\qquad |x|\leq s.\end{equation}
Moreover if $\psi$ satisfies WLSC($\beta,\theta,C^*$), there is a constant $C=C\left(d,\beta,C^*,\frac{\tilde{\psi}(\theta)}{\psi^*(\theta)}\right)$ such that, for $0<r<\theta^{-1}$, $s\leq r/2$,
$$P^x(|X_{\tau_{B(0,s)}}|\geq r)\leq C \left(\frac{s}{r}\right)^{\beta}, \qquad |x|\leq s.$$
\end{corollary}
\begin{proof}
Since $X_t$ is symmetric, by \cite{Schilling1998}, Remark 4.8, and Lemma \ref{fStarScalling},   we get
$$E^x\tau_{B_s}\leq E^x\tau_{B(x,2s)}\leq c_1  \frac{1}{\psi^*(s^{-1})},$$
where $c_1=c_1(d)$. Hence, the first claim follows by Lemma \ref{Harm1} and Corollary \ref{hApproxLemma}, while
the second claim is a consequence of Lemma \ref{star} and (\ref{JumpProb1}). We only have to check that $\tilde{\psi}(u)=\sup_{|x|=u}\psi(x)$ is positive on $(0,\infty)$. Suppose that there exists $u_0>0$, such that $\tilde{\psi}(u_0)=0$. Then, by subadditivity of $\sqrt{\psi}$, we have that $\tilde{\psi}(nu_0)=0$, for any $n\in\mathbb{N}$. Hence, by (\ref{tilde}),  $\tilde{\psi}(x)=0$, for any $|x|\geq u_0\vee \theta$. That implies that $\psi\equiv0$, what we exclude.
\end{proof}

\section{Isotropic Unimodal L\'{e}vy Processes}
In this section we assume that the process $X_t$ is isotropic unimodal. In the first subsection we obtain estimates for the  potential measure  and capacity of balls, which are essential for the rest of the paper. Next, we use them to get estimates for a potential kernel and Green function of the ball. The next subsection contains some improvements of these estimates in the  case of subordinate Brownian motions. Subsection 3.3 is devoted to prove the Harnack inequality and regularity estimates for harmonic functions. In the last subsection we give some examples.

By $\psi_0$, $\nu_0$ and $G_0$ we denote radial profiles of $\psi$, $\nu$ and $G$, respectively. For instance $\psi(x)=\psi_0(|x|)$.

\begin{lemma}{(\cite{Watanabe})}\label{LemmaWatanabe}
Let $X_t$ be a symmetric L\'{e}vy process, then the following conditions are equivalent:
\begin{description}
  \item[(1)] $X_t$ is isotropic unimodal.
  \item[(2)] $G^{\lambda}(dx)$ is isotropic unimodal, for  $\lambda> 0$ (in the transient case for $\lambda\geq 0$), where $G^{\lambda}(dx)=\Z e^{-\lambda t}P_t(dx)dt$ and $P_t(dx)= P^0(X_t\in dx)$.
  \item[(3)] $A=aI$, for some $a\geq 0$ and $\nu$ is isotropic unimodal.
\end{description}
\end{lemma}
Since $X_t$ is isotropic its distribution is supported by the whole space, so it is transient for $d\geq 3$. Notice that  $G(\{0\})>0$ if and only if $\psi$ is bounded.

In the following proposition we prove that the characteristic exponent of an isotropic unimodal L\'{e}vy process is almost increasing.
\begin{proposition}\label{psiStarPsi}We have, for any $x\in\Rd$,
$$\psi^*(|x|)\leq 12\psi(x).$$
\end{proposition}

\begin{proof}
Let us define, for $r\geq 0$,  $$\tilde{\psi}_0(r)=2\int_{0}^\infty\big[1-\cos(rz)\big]\nu_1(z)dz,$$ where $\nu_1(z)=\int_{\R^{d-1}}\nu_0(\sqrt{|w|^2+z^2})dw$. Then, by Lemma \ref{LemmaWatanabe}, we have $\psi_0(r)=a r^2+\tilde{\psi}_0(r)$, for some $a\geq 0$, where the factor $a r^2$ corresponds to the continuous part in the L\'evy decomposition of $X_t$.
Since $\nu_1$ is non-increasing on $(0,\infty)$,
\begin{eqnarray*}
\tilde{\psi}_0(r)&\geq& \sum^\infty_{k=0}\int^{(5/3\pi+2k\pi)/r}_{(\pi/3+2k\pi)/r}\nu_1(z)dz\geq \frac{4\pi}{3r}\sum^\infty_{k=0}\nu_1\left(\frac{5/3\pi+2k\pi}{r}\right)\\
&\geq& \frac{2}{3}\sum^\infty_{k=0}\int_{(5/3\pi+2k\pi)/r}^{(5/3\pi+2(k+1)\pi)/r}\nu_1(z)dz=\frac{2}{3}\nu_1\left[\frac{5\pi}{3r},\infty\right).
\end{eqnarray*}
{We also note that $1-\cos u\geq \frac{9}{2\pi^2}u^2$ if $|u|\le \pi/3$.}
We have, \begin{eqnarray*}\tilde{\psi}_0(r)&\geq& \frac{9}{\pi^2}\int^{1/r}_0 (zr)^2\nu_1(dz)+2\big[1-\cos(1)\big]\nu_1\left(\left[\frac{1}{r},\frac{5\pi}{3r}\right)\right)\\&\geq& \frac{9}{\pi^2}\left(\int^{1/r}_0 (zr)^2\nu_1(dz)+ \nu_1\left(\left[\frac{1}{r},\frac{5\pi}{3r}\right)\right)\right).\end{eqnarray*}
Hence,
$$\psi_0(r)\geq ar^2+\frac{1}{3}\Z  \left(1\wedge{(zr)^2}\right) \nu_1(z)dz .$$
Since  the function $\Z \left(1\wedge{(zr)^2}\right) \nu_1(z)dz$ is non-decreasing and $1-\cos u\leq 2(1\wedge u^2)$,
$$\psi^*(r)=\sup_{|x|\leq r}\psi(x) \leq ar^2+\sup_{s\leq r}4 \Z  \left(1\wedge{(zs)^2}\right) \nu_1(z)dz=ar^2+4\Z \left(1\wedge{(zr)^2}\right) \nu_1(z)dz.$$
Finally, we get $\psi^*(r)\leq 12 \psi_0(r)$.
\end{proof}

\subsection{Green function estimates}
\begin{lemma}\label{Laplace_potential_LP} Assume that $d\geq 3$. Let $f(r)=G\left(B_{\sqrt{r}}\right)$. There exists a constant $C_1=C_1(d)$ such that $$\frac{C_1}{\lambda\psi^*\left(\sqrt{\lambda}\right)}\leq \mathcal{L}f(\lambda)\leq \frac{36}{\lambda\psi^*\left(\sqrt{\lambda}\right)},\qquad\lambda>0.$$ \end{lemma}
\begin{proof}
 Since  $$e^{-|z|^2}= \frac1{({4\pi})^{d/2}}\int_{\R^d} {e^{i \left\langle z,x \right\rangle}}e^{-|x|^2/4}dx, \quad z\in \R^d,$$
 we have, for $\lambda>0$,
$$E e^{-\lambda |X_t|^2}= E  \frac1{({4\pi})^{d/2}}\int_{\R^d}  {e^{i\sqrt{\lambda}\left\langle x,X_t \right\rangle }}{e^{-|x|^2/4}}dx=\frac1{({4\pi})^{d/2}}\int_{\R^d}  {e^{-t\psi(\sqrt{\lambda} x  )}}e^{-|x|^2/4}dx.$$
Integrating with respect to $dt$ we have, $$\lambda{\cal L}f(\lambda)= \frac1{({4\pi})^{d/2}}\int_{\R^d} e^{-|x|^2/4}\frac{dx}{\psi(\sqrt{\lambda} x  )}.$$
By Proposition \ref{psiStarPsi}, $\frac{1}{12}\psi^*(|x|)\leq \psi(x)\leq \psi^*(|x|)$. Hence,
$$\frac{1}{({4\pi})^{d/2}}\int_{\R^d} e^{-|x|^2/4}\frac{dx}{\psi^*(\sqrt{\lambda} |x|  )}\leq \lambda{\cal L}f(\lambda)\leq \frac{12}{({4\pi})^{d/2}}\int_{\R^d} e^{-|x|^2/4}\frac{dx}{\psi^*(\sqrt{\lambda} |x|  )}.$$
By Lemma \ref{fStarScalling},
\begin{eqnarray*}
\frac{1}{2\psi^*(\sqrt{\lambda})}\frac{1}{1+|x|^2}&\leq &\frac{1}{\psi^*(\sqrt{\lambda} |x|  )}\leq\frac{2}{\psi^*(\sqrt{\lambda})}\frac{1+|x|^2}{|x|^2}.
\end{eqnarray*}
The above estimates imply $$\frac{C_1}{\psi^*(\sqrt{\lambda})}\leq \lambda{\cal L}f(\lambda)\leq \frac{36}{\psi^*(\sqrt{\lambda})},$$
where $C_1=\frac{1}{2^{d+1}\pi^{d/2}}\int_{\R^d}e^{-\frac{|x|^2}{4}}\frac{1}{1+|x|^2}dx$.
\end{proof}

\begin{proposition}\label{PotMeasureLP} Assume that $d\geq 3$. There is a constant $C_2=C_2(d)$ such that
$$\frac{C_2}{\psi^*(r^{-1})}\leq G\left(B_r\right)\leq  \frac{36e}{\psi^*(r^{-1})},\qquad r>0.$$
\end{proposition}
\begin{proof}Let $f(r)=G\left(B_{\sqrt{r}}\right)$.
Since $f(u)$ is non-decreasing
$$f(u)\leq \frac{e}{u}\int^\infty_ue^{-s/u}f(s)ds\leq \frac{e}{u}\mathcal{L}(f)(u^{-1}).$$
By Lemma \ref{Laplace_potential_LP} we get
\begin{equation}\label{PotMeasureUpper}G\left(B_r\right)\leq \frac{36e}{\psi^*(r^{-1})}.\end{equation}
By Lemma \ref{fStarScalling} we have, for $u\leq s$, $\frac{\psi^*(u^{-1/2})}{\psi^*(s^{-1/2})}\leq 4\frac{s}{u}$.
  Lemma \ref{Laplace_potential_LP} and (\ref{PotMeasureUpper}) give us, for $\kappa>1$,
\begin{eqnarray*}\int^\infty_{\kappa u}e^{-s/u}f(s)ds&\leq& 36e\int^\infty_{\kappa u}e^{-s/u}\frac{ds}{\psi^*(s^{-1/2})}\leq \frac{144e}{u\psi^*(u^{-1/2})}\int^\infty_{\kappa u}e^{-s/u}sds\\&=&\frac{144e\Gamma(2,\kappa)u}{\psi^*(u^{-1/2})}\leq c_1\Gamma(2,\kappa)\mathcal{L}(f)(u^{-1}),
\end{eqnarray*}
where $c_1=\frac{144e}{C_1}$.
Hence, for $\kappa$ such that $c_1\Gamma(2,\kappa)=\frac{1}{2}$ we have
$$\mathcal{L}(f)(u^{-1})\leq 2\int^{\kappa u}_{0}e^{-s/u}f(s)ds\leq 2uf(\kappa u).$$
Again, by Lemmas \ref{Laplace_potential_LP} and \ref{fStarScalling} we infer
$$G\left(B_r\right)\geq \frac{C_1}{2\psi^*(\sqrt{\kappa} r^{-1})}\geq \frac{C_1}{4(\kappa+1)\psi^*(r^{-1})}.$$
\end{proof}

By $\mathrm{Cap}^\lambda$, $\lambda\geq 0$, we denote the $\lambda$-capacity with respect to $X_t$. When $\lambda=0$ we omit a superscript $"0"$. For any non-empty compact set $A\subset \Rd$ there exists a measure $\rho_A$ (see e.g. \cite[Corollary II.8]{Bertoin}), called the equilibrium measure, which is supported by $A$ and
\begin{equation}\label{CapDef}G\rho_A(F)=\int G(F-x)\rho_A(dx)=\int_FP^x(T_A<\infty)dx,\quad F\subset \Rd.\end{equation}
Moreover $\rho_A(A)=\mathrm{Cap}(A)$. If the potential measure is absolutely continuous, then
$$G\rho_A(x)=P^x(T_A<\infty), \qquad  x\in \Rd.$$
\begin{proposition}\label{CapApprox}Let $d\geq 3$. There exists a constant $C=C(d)$ such that, for any $r\geq0$,
$$C^{-1}\psi^*(r^{-1})r^d\leq \mathrm{Cap}\left(\overline{B_r}\right)\leq C \psi^*(r^{-1})r^d.$$
\end{proposition}
\begin{proof}Since $d\geq 3$, $\mathrm{Cap}(\{0\})=0$, so we may  assume that $r>0$. By Lemma \ref{LemmaWatanabe}, $G$ is radially non-increasing.
Let $x\in \overline{B_r}$, then $$G\left(\overline{B_r}-x\right)=G(\{0\})+\int_{B_r}G(y-x)dy.$$
By \cite{SiSV}, Proposition 5.1, there exists a constant $c_1=c_1(d)$ such that
$$c_1\int_{B_r}G(y)dy\leq \int_{B_r}G(y-x)dy\leq  \int_{B_r}G(y)dy.$$
Hence
$$c_1G\left(\overline{B_r}\right)\leq G\left(\overline{B_r}-x\right)\leq G\left(\overline{B_r}\right).$$
By (\ref{CapDef}),
$$|\overline{B_r}|= G\rho_{\overline{B_r}}\left(\overline{B_r}\right).$$
This preparation  yields
 $$|\overline{B_r}|\leq G\left(\overline{B_r}\right) \mathrm{Cap}\left(\overline{B_r}\right)\leq c_1^{-1}|\overline{B_r}|,$$
which combined with Proposition \ref{PotMeasureLP} implies
\begin{equation}\label{CapForIULP} \frac{|B_1|}{36e}\psi^*(r^{-1})r^d\leq\mathrm{Cap}\left(\overline{B_r}\right)\leq \frac{|B_1|}{c_1C_2}\psi^*(r^{-1})r^d.\end{equation}
\end{proof}
\begin{remark}
Similar calculations provide the following estimates for $\lambda$-resolvent measure and $\lambda$-capacity
$$G^\lambda(B_r)\approx \frac{1}{\lambda+\psi^*(r^{-1})}\quad\mathrm{ and }\quad   \mathrm{Cap}^\lambda(\overline{B_r})\approx r^{d}\left({\lambda+\psi^*(r^{-1})}\right).$$
\end{remark}

By \cite{Watanabe}, Theorem 1, it is known that the smallest capacity among sets with the same volume is attained for a closed ball. 
\begin{corollary}There exists a constant $C_3=C_3(d)$ such that,
for any non-empty Borel set $A$,
\begin{equation}\label{CapGeneral}\mathrm{Cap}(A)\geq C_3\psi^*(|A|^{-1/d})|A|.\end{equation}
\end{corollary}
\begin{proof}
{
Let $r$ be such that $|A|=|B_r|$. By  Theorem 1 of \cite{Watanabe} and (\ref{CapForIULP})
$$\mathrm{Cap}(A)\geq \mathrm{Cap}(B_r)\geq  \frac{1}{36e}\psi^*(r^{-1})|A|.$$
Hence,  Lemma \ref{fStarScalling} implies
$$\mathrm{Cap}(A)\geq \frac{|B_1|^2}{72e(1+|B_1|^2)}\psi^*(|A|^{-1/d})|A|.$$
}
\end{proof}

To our best knowledge the following upper bound  for the potential kernel was known only for  subordinate Brownian motions and the lower one for special subordinate Brownian motions (see e.g. \cite{Zahle}). They were obtained as a consequence of appropriate estimates for the potential measure and the potential kernel of the subordinator, respectively.
\begin{theorem}\label{PotentialKernelLP}
Let $d\geq 3$.
Then there exists a constant $C_4=C_4(d)$ such that
$$G(x) \leq\frac{C_4}{|x|^d\psi^*(|x|^{-1})},\qquad x\in\R^d.$$
If additionally $\psi$ satisfies WLSC$(\beta,R^{-1},C^*)$, then
$$\frac{C_5}{|x|^d\psi^*(|x|^{-1})}\leq G(x), \qquad |x|\leq bR,$$
where  $b= c(d,\beta)(C^*)^{1/\beta}<1$ and $C_5=c(d) b^{d}$.
\end{theorem}
\begin{proof}
Since $G$ is radially non-increasing,
$$\int_{B(0,|x|)}G(y)dy\geq |B_1||x|^dG(x).$$
By Proposition \ref{PotMeasureLP}
$$  |x|^d G(x) \leq \frac{36e}{|B_1|}\frac{1}{\psi^*(|x|^{-1})},$$
which completes the proof of the upper bound.

Again, by radial monotonicity, we have, for $\kappa>1$,
\begin{eqnarray*}
G( x)\geq \frac{G\left(B_{\kappa|x|}\setminus B_{|x|}\right)}{\left|B_{\kappa|x|}\setminus B_{|x|}\right|}\geq \frac{1}{|B_1|\kappa^{d}}\frac{G\left(B_{\kappa|x|}\setminus B_{|x|}\right)}{|x|^d}.
\end{eqnarray*}
Suppose that $\psi$ satisfies WLSC$(\beta,R^{-1},C^*)$,  then by Propositions \ref{PotMeasureLP} and \ref{psiStarPsi}, for $\kappa |x|\leq R$,
\begin{eqnarray*}
G\left(B_{\kappa|x|}\setminus B_{|x|}\right)&=&G\left(B_{\kappa|x|}\right)-G\left(B_{|x|}\right)\geq \frac{C_2}{\psi^*\left((\kappa|x|)^{-1}\right)}-\frac{36e}{\psi^*( |x|^{-1})}\nonumber \\
&=& \frac{36e}{\psi^*(|x|^{-1})}\left(c_1\frac{\psi^*(|x|^{-1})}{\psi^*((\kappa |x|)^{-1})}-1\right)\geq \frac{36e}{\psi^*(|x|^{-1})}\left(\frac{c_1C^*}{12}\kappa^{\beta}-1\right).
\end{eqnarray*} Hence, for $\kappa=(24/(c_1C^*))^{\frac{1}{\beta}}$, we get
$$G(x)\geq \frac{36e}{|B_1|\kappa^d }\frac{1}{|x|^{d}\psi^*(|x|^{-1})},\qquad |x|\leq R/\kappa.$$
\end{proof}

\begin{corollary}\label{GreenPotentialGreenFunction}Let $d\geq 3$. If WLSC$(\beta,R^{-1},C)$ holds for $\psi$, then there exists a constant $b< 1$ such that
$$G(x)\approx G(B_{|x|})|x|^{-d},\qquad |x|\leq bR.$$
\end{corollary}
The above comparability is crucial in our proof of the scale invariant Harnack inequality. In the next subsection we show the converse of Corollary \ref{GreenPotentialGreenFunction} in the case of special subordinate Brownian motions.

\begin{remark}\label{doubling}
If $\psi$ satisfies WLSC$(\beta,R^{-1},C^*)$, then a local doubling condition for $G$ holds. This means there exists  constant $C=C(d,\beta,C^*)$ such that $CG(x) \leq G(2x)\leq G(x)$, for $0<|x|\leq bR/2$, where a constant $b$ is from Theorem \ref{PotentialKernelLP}.
\end{remark}

Standard arguments provide the following proposition.
\begin{proposition}\label{GreenLowerLP}Let $d\geq 3$. Suppose that $\psi$ satisfies WLSC$(\beta,R^{-1},C^*)$, then, for any $\varepsilon \in (0,1)$, there exists a constant $L=L(\varepsilon,d,\beta,C^*)>1$ such that, for $r\leq R$,
$$G_{B_r}(x,y)\geq \varepsilon G(x-y), \qquad L|x-y|\leq (r-|x|)\vee (r-|y|).$$
\end{proposition}
\begin{proof}
Since $\psi$ satisfies WLSC, it is unbounded. Therefore $G(\{0\})=0$ and due to Lemma \ref{LemmaWatanabe} the potential measure is absolutely continuous. We may and do assume that $|y|\leq |x|$ and $L\geq b^{-1}$, where $b$ appears in Theorem \ref{PotentialKernelLP}.  Since $|X_{\tau_{B_r}}-y|\geq r-|y|\geq L|x-y|$, by radial monotonicity of $G$ and (\ref{Hunt1}),
$$G_{B_r}(x,y)=G(x-y)-E^xG(X_{\tau_{B_r}}-y)\geq G(x-y)-G(L(x-y)).$$
By this, Theorem \ref{PotentialKernelLP} and Lemma \ref{fStarScalling},
\begin{eqnarray*}
G_{B_r}(x,y)&\geq&G(x-y)\left(1-\frac{C_4}{C_5L^d}\frac{\psi^*(|x-y|^{-1})}{\psi^*((L|x-y|)^{-1})}\right)\geq G(x-y)\left(1-\frac{4C_4}{C_5L^{d-2}}\right).
\end{eqnarray*}
Hence, for $L=\left(\frac{4C_4}{C_5(1-\varepsilon)}\right)^{1/(d-2)}\vee b^{-1}$ we obtain
$$G_{B_r}(x,y)\geq \varepsilon G(x-y).$$
\end{proof}

\subsection{Subordinate Brownian motions}
In this subsection we improve Theorem \ref{PotentialKernelLP} and Proposition \ref{GreenLowerLP} in the case of subordinate Brownian motions. Namely,  we  prove that $b=1$ and $L=2$, for some $\varepsilon>0$. We assume in this subsection that $X_t$ is a subordinate Brownian motion.

The following lemma is well known (see e.g. \cite{Hawkes1975}), but for the convenience of the reader we prove it with a short and simple proof.
\begin{proposition}\label{potential_subordinator}
For $r>0$,
$$\frac{1-2e^{-1}}{2\phi(r^{-1})} \leq U[0,r)\leq \frac{e}{\phi(r^{-1})}.$$
\end{proposition}
\begin{proof}
Notice that for $\lambda>1$, $\phi(\lambda r)\leq \lambda \phi(r)$. Hence,
\begin{eqnarray*}
\frac{1}{2\phi((2r)^{-1})}&\leq&\frac{1}{\phi(r^{-1})}=\int^{2r}_0e^{-r^{-1}t}U(dt)+\int^\infty_{2r}e^{-r^{-1}t}U(dt)\\&\leq& U[0,2r)+e^{-1}\int^\infty_{2r}e^{-(2r)^{-1}t}U(dt)
\leq U[0,2r)+\frac{e^{-1}}{\phi((2r)^{-1})},
\end{eqnarray*}
which proves the lower bound.

On the other hand
$$U[0,r)\leq e\int^r_0e^{-r^{-1}t}U(dt)\leq \frac{e}{\phi(r^{-1})}.$$
\end{proof}

The following theorem is an improvement of Theorem \ref{PotentialKernelLP}. Such result is known (see e.g. \cite{Zahle}, Theorem 1), under an additional assumption   that $\phi$ is a special Bernstein function.
\begin{theorem}\label{PotentialKernel}Let $d\geq 3$ and $X_t$ be a subordinate Brownian motion.
If $\phi$ satisfies WLSC$(\beta,R^{-2},C^*)$, then there exists a constant $C_6=C_6(d,\beta,C^*)$ such that
$$G(x)\geq \frac{C_6}{|x|^d\phi(|x|^{-2})}, \qquad |x|\leq R.$$
\end{theorem}
\begin{proof}
Let $\kappa<1$.  By (\ref{PotentialDensitySub}) we have
\begin{eqnarray*}
G(x)\geq \int^{|x|^2}_{\kappa |x|^2}g_t(x)U(dt)\geq \left(g_{\kappa}(\mathbf{1})\wedge g_{1}(\mathbf{1})\right)|x|^{-d}U[\kappa|x|^2,|x|^2),
\end{eqnarray*}
where $\mathbf{1}=(1,0,\ldots,0)$.
Suppose that $\phi$ satisfies WLSC$(\beta, R^{-2},C^*)$, then by Lemma \ref{potential_subordinator}, for $|x|\leq R$,
\begin{eqnarray*}
U[\kappa |x|^{2},|x|^{2})&=&U[0,|x|^{2})-U[0,\kappa |x|^2)\geq \frac{1-2e^{-1}}{2\phi(|x|^{-2})}-\frac{e}{\phi((\kappa |x|^2)^{-1})}\nonumber \\
&=& \frac{1-2e^{-1}}{2\phi(|x|^{-2})}\left(1-c_1\frac{\phi(|x|^{-2})}{\phi((\kappa |x|^2)^{-1})}\right)\geq \frac{1-2e^{-1}}{2\phi(|x|^{-2})}\left(1-c_2\kappa^{\beta}\right),\label{PotentialSuborDifference}
\end{eqnarray*} where $c_2=\frac{2e^2}{(e-2)C^*}$. Hence, for $\kappa=(2c_2)^{-\frac{1}{\beta}}$, we get
$$G(x)\geq \frac{c_3}{|x|^{d}\phi(|x|^{-2})},\qquad |x|\leq R,$$
where $c_3=\frac{1-2e^{-1}}{4}\left(g_{\kappa}(\mathbf{1})\wedge g_{1}(\mathbf{1})\right)$.
\end{proof}

The following theorem is a converse of the above theorem (and Corollary \ref{GreenPotentialGreenFunction}) in the case of special subordinate Brownian motions. Since the comparability in Corollary  \ref{GreenPotentialGreenFunction} is the key ingredient in the proof of the Krylov-Safonov estimate  it seems that the approach of Bass and Levin  for proving the Harnack inequality  can not be used if $\phi$ does not satisfy WLSC.
\begin{theorem}
Let $d\geq 3$ and  $X_t$ be a special subordinate Brownian motion. There exists a constant $C$ such that $G(x)\geq \frac{C}{|x|^d\phi(|x|^{-2})}$, for $|x|\leq R$ if and only if  $\phi$ satisfies WLSC$(\beta,R^{-2},1)$, for some $\beta>0$.
\end{theorem}
\begin{proof}Due to Theorem \ref{PotentialKernel} it is enough to show that the existence of a constant $c_1$ such that
 \begin{equation}\label{SBF_Greenmeasure1}G(x)\geq \frac{c_1}{|x|^d\phi(|x|^{-2})},\qquad |x|\leq R,\end{equation}
implies the weak lower scaling condition for $\phi$.  Suppose that (\ref{SBF_Greenmeasure1}) holds. Since the process is transient
$$\int_{B_R}G(x)dx\leq G(B_R)<\infty,$$
 which combined with (\ref{SBF_Greenmeasure1}) shows that  $\phi$ is unbounded and consequently  the potential measure of $X_t$ is absolutely continuous. Since $\phi$ is a special Bernstein function, by (\ref{PotentialDensitySub}), we have
 $$G(x)=\Z g_s(x)u(s)ds,$$
 where $u$ is non-increasing.
By (\ref{PotMeasureSubLaplace}),
\begin{eqnarray}
\left|\left(\frac{1}{\phi}\right)'\right|\left(\lambda\right)&=&\int^\infty_0se^{-\lambda s}u(s)ds\geq u(\lambda^{-1})\int^{\lambda^{-1}}_0se^{-\lambda s}ds\nonumber\\&=&(1-2e^{-1})\lambda^{-2}u(\lambda^{-1}).\label{1/phiPrime}
\end{eqnarray}
Since $\frac{1}{\phi}$ is completely monotone $\left|\left(\frac{1}{\phi}\right)'\right|$ is non-increasing. Due to (\ref{1/phiPrime}), monotonicity of $\left|\left(\frac{1}{\phi}\right)'\right|$ and $u$ implies
$$u(s)\leq \frac{1}{1-2e^{-1}}\left|\left(\frac{1}{\phi}\right)'(|x|^{-2})\right|\left(|x|^{-4}\vee s^{-2}\right).$$
Hence,
\begin{eqnarray*}
G(x)&\leq& \frac{1}{1-2e^{-1}}\left|\left(\frac{1}{\phi}\right)'(|x|^{-2})\right|\Z\left(|x|^{-4}\vee s^{-2}\right)g_s(x)ds=c_2\frac{\left(-\frac{1}{\phi}\right)'(|x|^{-2})}{|x|^{d+2}},
\end{eqnarray*}
where $c_2=\frac{16\Gamma\left(d/2+1,\frac{1}{4}\right)+\gamma\left(d/2-1,\frac{1}{4}\right)}{4\pi^{d/2}(1-2e^{-1})}$.
Using  (\ref{SBF_Greenmeasure1}),
$$\frac{1}{\phi}(\lambda)\leq \frac{c_2}{c_1}\left(-\frac{1}{\phi}\right)'(\lambda)\lambda,\quad \mathrm{ for }\,\,\lambda\geq R^{-2}.$$
 That is, for $\lambda\geq R^{-2}$,  $\phi(\lambda)\leq \frac{c_2}{c_1}\phi'(\lambda)\lambda$. Let $\beta< \frac{c_1}{c_2}$, then the function
$\phi(u)u^{-\beta}$ is increasing on $[R^{-2},\infty)$, since $\left(\phi(u)u^{-\beta}\right)'> 0$. In consequence
$$\frac{\phi(\lambda u)(\lambda u)^{-\beta}}{\phi(u)u^{-\beta}}\geq 1,\qquad \lambda\geq 1, u\geq  R^{-2},$$
which completes the proof.
\end{proof}

Let $D$ be an open set and $x\in D$. Denote $\delta_D(x)$ a distance $x$ from a boundary of $D$. The following theorem improves Proposition \ref{GreenLowerLP}. Like in the case of  Theorem \ref{PotentialKernel} such result was  known only for special subordinate Brownian motions for which the characteristic exponent or its derivative satisfies some scaling conditions (see e.g. \cite{KSV}, \cite{KM}).  These results were obtained by  standard arguments we used in  Proposition \ref{GreenLowerLP}, therefore the appropriate constants depend  on a process.  Our proof for a  special subordinate Brownian motion does not require any scaling properties and the appearing constant depends only on the dimension.
\begin{theorem}Let $d\geq 3$ and $D$ be an open set.
Suppose that $\phi$ is a unbounded special Bernstein function,  then \begin{equation}\label{GreenLower}G_{D}(x,y)\geq C_7 G(x-y),\qquad  2|x-y|\leq \delta_D(x)\vee \delta_D(y).\end{equation}
where $C_7=\frac{\Gamma\left(\frac{d}{2}-1,\frac{1}{4}\right)}{\Gamma\left(\frac{d}{2}-1\right)}\left(1-e^{-\frac{3}{4}}\right)$.

If $\phi$ is only a Bernstein function but satisfies WLSC$(\beta,R^{-2},C^*)$, then (\ref{GreenLower}) holds, if additionally $|x-y|\leq R$ with  a constant $C=C(d,\beta,C^*)$ instead of $C_7$.
\end{theorem}
\begin{proof}Let us assume that $\delta_D(x)\leq \delta_D(y)$ and $x\neq y$. Since $2|x-y|\leq \delta_D(y)\leq |X_{\tau_D}-y|$, by (\ref{Hunt1}) and radial monotonicity,
\begin{equation}\label{GreenLower1}G_D(x,y)=G(x,y)-E^xG(X_{\tau_D}-y)\geq G(x-y)-G(2(x-y)).\end{equation}
Let us define a function $$K_p(q)=\int^\infty_0(4\pi s)^{-d/2}\left(e^{-\frac{p^2}{4s}}-e^{-\frac{q^2}{4s}}\right)U(ds),\qquad p,q\geq 0.$$
Due to (\ref{PotentialDensitySub}) we have
\begin{equation}\label{GreenLower2}G(x-y)-G(2(x-y))=K_{|x-y|}(2|x-y|).\end{equation}
Moreover,
\begin{eqnarray}
K_{|x-y|}(2|x-y|)&\geq& \int^{|x-y|^2}_0(4\pi s)^{-d/2}e^{-\frac{|x-y|^2}{4s}}\left(1-e^{-\frac{3|x-y|^2}{4s}}\right)U(ds)\nonumber\\
&\geq& \left(1-e^{-\frac{3}{4}}\right) \int^{|x-y|^2}_0(4\pi s)^{-d/2}e^{-\frac{|x-y|^2}{4s}}U(ds).\label{GreenLower3}
\end{eqnarray}

Suppose that  $\phi$ is a special Bernstein function. Then there exists a non-increasing function $u$ such that $U(ds)=u(s)ds$. Monotonicity  of $u$ yields
\begin{eqnarray*} \int_{|x-y|^2}^\infty(4\pi s)^{-d/2}e^{-\frac{|x-y|^2}{4s}}u(s)ds&\leq& \frac{\gamma\left(\frac{d}{2}-1,\frac{1}{4}\right)}{4\pi^{d/2}}u(|x-y|^2)|x-y|^{2-d}\\&\leq& \frac{\gamma\left(\frac{d}{2}-1,\frac{1}{4}\right)}{\Gamma\left(\frac{d}{2}-1,\frac{1}{4}\right)}\int^{|x-y|^2}_0(4\pi s)^{-d/2}e^{-\frac{|x-y|^2}{4u}}u(s)ds.\end{eqnarray*}
Hence,
$$G(x-y)\leq \left(1+\frac{\gamma\left(\frac{d}{2}-1,\frac{1}{4}\right)}{\Gamma\left(\frac{d}{2}-1,\frac{1}{4}\right)}\right) \int^{|x-y|^2}_0(4\pi s)^{-d/2}e^{-\frac{|x-y|^2}{4s}}u(s)ds\leq c_2 K_{|x-y|}(2|x-y|),$$
where $c_2=\frac{\Gamma\left(\frac{d}{2}-1\right)}{\Gamma\left(\frac{d}{2}-1,\frac{1}{4}\right)\left(1-e^{-\frac{3}{4}}\right)}$.
In consequence
$$G_D(x,y)\geq c_2^{-1} G(x-y),\qquad 2|x-y|\leq \delta_D(y).$$

Now, let us suppose that $\phi$ is only Bernstein function (it is no longer  assumed  that it is special) satisfying WLSC$(\beta,R^{-2},C^*)$. Then, by the proof of Theorem \ref{PotentialKernel} and Theorem \ref{PotentialKernelLP} there exists  a constant $\kappa<1$,  such that
$$\int^{|x-y|^2}_{\kappa|x-y|^2}g_s(x-y)U(ds)\geq \frac{C_6}{C_4} G(x-y), \qquad |x-y|\leq R.$$
Hence, by (\ref{GreenLower1})--(\ref{GreenLower3})
$$G_{D}(x,y)\geq \left(1-e^{-\frac{3}{4}}\right)\frac{C_6}{C_4} G(x-y).$$

\end{proof}

\begin{remark}\label{RemarkGreenPotential}Let $d\geq 3$. Suppose that $\phi$ is an unbounded special Bernstein function
then there exists a constant $C=C(d)$ such that, for any $r>0$
$$CG(x-y)\leq G_{B_r}(x,y)\leq G(x-y),\qquad x,y\in B(0,r/5).$$
 If $\phi$ is only a Bernstein function satisfying WLSC$(\beta,R^{-2},C^*)$, then there exists a constant $C=C(d,\beta, C^*)$ such that, for any $r\leq R$ the above inequality holds.
\end{remark}

\subsection{Harnack inequality and H\"{o}lder regularity}
The goal of this subsection is to prove the main results of this paper, that is Theorems \ref{Harnack} and \ref{Holder}. In this subsection we assume that $X_t$ is an isotropic unimodal L\'{e}vy process with the characteristic exponent satisfying WLSC$(\beta,\theta, C^*)$. Since $\psi$ satisfies the weak lower scaling condition, therefore it is unbounded. Hence, the potential measure is absolutely continuous. Let $R=\theta^{-1}$. By $L$ we denote the constant from Proposition \ref{GreenLowerLP} for $\varepsilon=C_7$ or $L=2$ in the case of special subordinate Brownian motions and let $r_0=\frac{r}{2L+1}$. Notice that by the proof of Proposition \ref{GreenLowerLP}, $2r_0\leq bR$, where $b$ is from Theorem \ref{PotentialKernelLP}.

In the proof of the following proposition we follow closely  the ideas of \cite{BogdanSztonyk}, where symmetric stable L\'{e}vy processes were considered.
\begin{proposition}\label{PropSuppportOut}Let $d\geq 3$ and $\psi$ satisfy WLSC$(\beta,R^{-1},C^*)$.
 Then there exists a constant $C_8=C(d,\beta,C^*)$ such that, for any $r\leq R$, and any non-negative function $H$ such that $\mathrm{supp} H \subset \overline{B_r}^c$,
$$E^xH(X_{\tau_{B_{r_0}}})\leq C_8 E^yH(X_{\tau_{B_r}}),\qquad x,y\in B_{\frac{r_0}{2}}.$$
\end{proposition}
\begin{proof}
Due to Lemma \ref{LemmaWatanabe}, the Ikeda-Watanabe formula (\ref{IkedaWatanabe}) and (\ref{Hunt1}) we obtain that the Poisson kernel of  $B_r$ exists  and
$$E^yH(X_{\tau_{B_r}})=\int_{B_r^c}H(z)P_{B_r}(y,z)dz, $$
where
$$P_{B_r}(y,z)=\int_{B_r}G_{B_r}(y,w)\nu(z-w)dw.$$
Hence, it is enough to prove there is a constant $c_1$ such that
 \begin{equation}\label{Supp1}P_{B_{r_0}}(x,z)\leq c_1 P_{B_{r}}(y,z),\qquad x,y\in B_{\frac{r_0}{2}},\,|z|>r.\end{equation}
By Proposition  \ref{GreenLowerLP} and radial monotonicity of $G$,
\begin{equation}\label{Supp2}P_{B_r}(y,z)\geq \int_{B_{r_0}}G_{B_r}(y,w)\nu(z-w)dw\geq C_7G_0\left(2r_0\right)\int_{B_{r_0}}\nu(z-w)dw.\end{equation}
Since $\nu$ is radially non-increasing, for $w\in B_{\frac{3}{4}r_0}$,
$$\label{nuUpperBound}\nu(w-z)\leq \frac{\int_{B\left(w_z,\frac{r_0}{8}\right)}\nu(u-z)du}{|B_{\frac{r_0}{8}}|}\leq c_2r_0^{-d}\int_{B_{r_0}}\nu(u-z)du,$$ where $w_z=w+\frac{r(z-w)}{8|z-w|}$ and $c_2=\frac{8^d}{|B_1|}$. Hence, by a doubling condition (see Remark \ref{doubling})  and radial monotonicity of $G$,
\begin{eqnarray*}
P_{B_{r_0}}(x,z)&\leq&\int_{B_{r_0}}G(x-w)\nu(w-z)dw\\&\leq& G_0\left(\frac{r_0}{4}\right)\int_{B_{r_0}\setminus B_{\frac{3}{4}r_0}}\nu(w-z)dw\\ &&+ c_2r_0^{-d}\int_{B_{r_0}}\nu(u-z)du\int_{ B_{\frac{3}{4}r_0}}G(x-w)dw\\
&\leq&  \left( c_3G_0(2r_0)+c_2r_0^{-d}\int_{ B\left(x,\frac{5}{4}r_0\right)}G(x-w)dw\right)\int_{B_{r_0}}\nu(u-z)du\\&\leq& c_4\left( {G}_0(2r_0)+(2r_0)^{-d}G(B_{2r_0})\right)\int_{B_{r_0}}\nu(u-z)du,
\end{eqnarray*}
where $c_4=c_3\vee (2^dc_2)$. Corollary \ref{GreenPotentialGreenFunction} provides
$$P_{B_{r_0}}(x,z)\leq c_5 {G}_0(2r_0) \int_{B_{r_0}}\nu(u-z)du,$$
for some constant $c_5=c_5(d,\beta,C^*)$. Due to  (\ref{Supp2}) this implies  (\ref{Supp1}), which completes the proof.
\end{proof}

In the following proposition we prove the Krylov-Safonov estimate, which is crucial for proving the scale invariant Harnack inequality. In the proof we use some ideas of \cite{SiSV}, Lemma 6.2.
\begin{proposition}\label{PropKrylovSafonov}Let $d\geq 3$ and $\psi$ satisfy WLSC$(\beta,R^{-1},C^*)$.  There exists a constant $C_9=C(d,\beta,C^*)$ such that for any $r\leq R$ and any compact $A\subset B_{r_0}$,
$$P^x(T_{A}<\tau_{B_r})\geq C_9\frac{|A|}{|B_{r_0}|},\qquad x\in B_{r_0}.$$
\end{proposition}
\begin{proof} Let $B$ be an open set and $A\subset B$ be compact. Similarly to (\ref{CapDef}), let
$$G_B\rho_A(x)=\int_A G_B(x,y)\rho_A(dy).$$
Then, by the strong Markov property and (\ref{Hunt1}),
\begin{eqnarray*}
G_B\rho_A(x)&=&G\rho_A(x)-E^xG\rho_A(X_{\tau_B})=P^x(T_A<\infty)-E^xP^{X_{\tau_B}}(T_A<\infty)\\
&\leq&P^x(T_A<\infty)-E^x[P^{X_{\tau_B}}(T_A<\infty),T_A>\tau_B]= P^x(T_A<\tau_B).
\end{eqnarray*}
On the other hand
\begin{eqnarray*}
G_B\rho_A(x)&\geq&\inf_{y\in A}G_B(x,y)\rho_A(A).
\end{eqnarray*}
This implies
\begin{equation}\label{CapSetA}P^x(T_{A}<\tau_{B})\geq \inf_{z\in A}G_{B}(x,z) \mathrm{Cap}(A).\end{equation}
By Proposition  \ref{GreenLowerLP} and radial monotonicity of $G$, for $x\in B_{r_0}$,
$$\inf_{z\in B_{r_0}}G_{B_r}(x,z)\geq C_7 \inf_{z\in B_{r_0}}G(x-z)\geq C_7 G_0(2r_0).$$
By Theorem \ref{PotentialKernelLP}, $ G_0(2r_0) \geq  \frac{C_5}{(2r_0)^d\psi^*((2r_0)^{-1})}$. Combining this with  (\ref{CapSetA})  for $A\subset B= B_{r}$, and (\ref{CapGeneral}) we obtain
$$P^x(T_{A}<\tau_{B_r})\geq c_1\frac{ |A|\psi^*(|A|^{-1/d})}{r_0^d\psi^*((2r_0)^{-1})},$$
where $c_1=\frac{C_3C_5C_7}{2^d}$.
By  Lemma \ref{fStarScalling}, for $A\subset B_{r_0}$, there exists a constant $c_2=c_2(d)$ such that
$\frac{ \psi^*(|A|^{-1/d})}{\psi^*((2r_0)^{-1})}\geq c_2$. Hence,
$$ P^x(T_{A}<\tau_{B_r})\geq c_1c_2|B_1| \frac{|A|}{|B_{r_0}|}.$$
\end{proof}
Let us notice that until now,  under the assumption that $X_t$ is isotropic unimodal with its characteristic exponent satisfying WLSC$(\beta,R^{-1},C^*)$, all  the constants that appear in the paper  depend only on $d$, $\beta$, $C^*$.  None of them  depends on $R$ or  $\theta$, respectively.

Now, we are ready to prove the main results of our paper.
\begin{proof}[Proof of Theorem \ref{Harnack}]
We prove the result for bounded harmonic functions. The boundedness assumption   can be removed in a  similar way as in \cite{SongVondracek}, Theorem 2.4. Assume that $\psi$ satisfy WLSC$(\beta,\theta,C^*)$. Let $R_0>0$. We prove that there exists a constant $c_1=c_1(R_0)$ such that, for any function $h$ non-negative on $\Rd$ and  harmonic in a ball $B_r$, $r\leq R_0$, \begin{equation}\label{HI}\sup_{x\in B(0,r/2)}h(x)\leq C\inf_{x\in B(0,r/2)}h(x).\end{equation}

 Recall that $R=\theta^{-1}$. With  Propositions \ref{PropSuppportOut}  and \ref{PropKrylovSafonov} at hand   we can use the approach of Bass and Levin (\cite{BL})   to get the existence of constants $c_2=c_2(d,\beta,C^*)$ and $a=a(d,\beta,C^*)<1$ such that, for  any function $h$ non-negative and bounded on $\Rd$ and  harmonic in a ball $B_r$, $r\leq R$,
$$\sup_{x\in B_{ar}}h(x)\leq c_2\inf_{x\in B_{ar}}h(x).$$

Next, we use the standard chain argument to get
 $$\sup_{x\in B_{r/2}}h(x)\leq c_3\inf_{x\in B_{r/2}}h(x),$$
 where $c_3=c_3(d,c_2,a)$. If $R_0\leq R$ we have (\ref{HI}). Notice, that if $\psi$ satisfies the global weak lower scaling condition ($R=\infty$)  we get the global scale invariant Harnack inequality, since we can take $c_1=c_3$ and $c_3$ does not depend on $R_0$.
For $R_0>R$, one can use again the chain argument to get (\ref{HI}), for any harmonic function on $B_r$, $r\leq R_0$. But then the constant $c_1=c_1\left(d,c_3,\frac{R_0}{R}\right)$.
\end{proof}
To deal with dimension $d\le 2$ we use the idea from \cite{Mimica2012}, which relies on extending harmonic functions to higher dimensional spaces.

\begin{corollary}{\label{proj}}Let $d\leq2$. Suppose that $\psi$ satisfies WLSC$(\beta,\theta,C^*)$ and
 there exists an unimodal isotropic L\'{e}vy process $Y_t\in \R^3$, such that  $X_t$  is a projection of  $Y_t$. Then  the scale invariant Harnack inequality holds.
\end{corollary}
\begin{proof}We  present only the one-dimensional case. Without loss of generality we can assume that $X_t=Y_t^{(1)}$, where $Y_t=(Y^{(1)}_t,Y^{(2)}_t, Y^{(3)}_t)$. Suppose that  $h$ is  harmonic and non-negative with respect to $X_t$ in $(-r,r)$, then by the strong Markov property a function $f:\R^3\to [0,\infty)$ defined by $f(x^{(1)},x^{(2)}, x^{(3)})=h(x^{(1)})$ is harmonic with respect to $Y_t$ in $(-r,r)\times \R^2$. Since $X_t, Y_t$ are isotropic then the characteristic exponent of  $Y_t$ denoted by $\psi^Y$  satisfies   $\psi^Y(x)=\psi_0(|x|)$. We recall that $\psi_0$ is the radial profile of  $\psi$. Hence, $\psi^Y$ satisfies WLSC$(\beta,R^{-1},C^*)$. Due to Theorem \ref{Harnack} the scale invariant Harnack inequality holds for $Y_t$, so it must  hold  for $X_t$.
\end{proof}

\begin{proof}[Proof of Theorem \ref{Holder}]
 With the Krylov-Safonov type estimate (Proposition \ref{PropKrylovSafonov}) and the second part of Corollary \ref{JumpProb} the  proof is similar to the proof in \cite{BL}, Theorem 4.1, therefore it is omitted.
\end{proof}
\begin{corollary}Let $d\leq2$. Suppose that $\psi$ satisfies WLSC$(\beta,\theta,C^*)$ and
 there exists an unimodal isotropic L\'{e}vy process $Y_t\in \R^3$, such that  $X_t$  is a projection of  $Y_t$. Then the conclusion of Theorem \ref{Holder} holds for $X_t$.
\end{corollary}

Let us remark that in general we can not find an  isotropic L\'{e}vy process  $Y_t$ in higher dimension such that
  $X_t$   is a projection of $Y_t$. If there exists such process $Y_t$, then $\nu_0(|x|)=\int_{\R^{3-d}}\nu_0^Y(\sqrt{|x|^2+|y|^2})dy$. Hence $\nu_0$  must be continuous on $(0,\infty)$.  Therefore if $\nu_0$ is not continuous the construction of $Y_t$ is impossible.

	On the other hand any Bernstein function defines a subordinate Brownian motion in every dimension  hence  the following theorem  holds with no  restriction on dimension.
\begin{theorem}\label{Harnack_SBM}
Let $d\geq 1$ and $X_t$ be a subordinate Brownian motion. Suppose that $\phi$ satisfies WLSC$(\beta,\theta,C^*)$.
Then the scale invariant Harnack inequality as well as   the conclusion of Theorem \ref{Holder} hold. Moreover, if  $\phi$ satisfies the global weak lower scaling condition, then the global scale invariant Harnack inequality holds.
\end{theorem}

\subsection{Examples}
We begin with a result which is helpful in verifying the scaling conditions for the characteristic exponent.
\begin{proposition}\label{muAsymbol}Let $X_t$ be a L\'{e}vy process with a L\'{e}vy measure $\nu(dx)=\nu(x)dx$. Suppose that for some $r\in (0,\infty]$,  a  constant $c_1$ and a non-increasing function $f:(0,\infty)\to [0,\infty)$,
$$c_1^{-1}\frac{f(|x|)}{|x|^d}\leq \nu(x)\leq c_1\frac{f(|x|)}{|x|^d},\quad |x|<r.$$
If there exist $c_2\geq 1$ and $\beta>0$ such that $$f(\lambda s)\leq c_2\lambda^{-\beta}f(s), \quad \lambda>1\,\, \lambda s\leq r,$$
then $\psi$ satisfies WLSC($\beta,r^{-1},C^*$) for some $C^*$.
\end{proposition}
\begin{proof}
Let $\nu_0(s)=c_1^{-1}\frac{f(s)}{s^d}\textbf{1}_{(0,r)}(s)$ and $Y_t$ be an isotropic unimodal L\'{e}vy process with L\'{e}vy measure $\nu^Y(dx)=\nu_0(|x|)dx$.  By Corollary \ref{hApproxLemma} and Proposition \ref{psiStarPsi},
$$\psi^Y(x)\approx \int_{B_r}(1\wedge (|x||z|)^2 )f(|z|)\frac{dz}{|z|^d}, \quad x\in \Rd.$$
For $\lambda> 1$ and $x\in \Rd$ we have
\begin{eqnarray*}
\psi^Y(\lambda x) &\approx&\int_{B_{\lambda r}}(1\wedge (|x||z|)^2 )f(|z|/\lambda)\frac{dz}{|z|^d}
\geq\int_{B_r}(1\wedge (|x||z|)^2 )f(|z|/\lambda)\frac{dz}{|z|^d}\\
&\geq&c_2^{-1}\lambda^{\beta}\int_{B_r}(1\wedge (|x||z|)^2 )f(|z|)\frac{dz}{|z|^d}\approx \lambda^{\beta}\psi^{Y}(x).
\end{eqnarray*}
That is $\psi^Y$ satisfies the global WLSC. Moreover,
$$\psi^Y(x)\leq \psi(x)\leq c_1^2\psi^Y(x)+\int_{B^c_r}\nu(z)dz, \quad x\in \Rd.$$
Hence, if $r=\infty$ we get that $\psi$ satisfies WLSC($\beta,0,C^*$) for some $C^*$. If $r<\infty$, by Proposition \ref{psiStarPsi}, for $|x|\geq r^{-1}$,
$$\psi^Y(x)\geq \psi^Y(r^{-1})/12\geq c(r)\int_{B^c_r}\nu(z)dz,$$
what ends the proof.
\end{proof}

Recall that $X_t$ is an isotropic unimodal L\'{e}vy process and $\nu(dx)=\nu_0(|x|)dx$. If we do not mention otherwise, then  $d\geq 3$.
\begin{example}
Let $A=0$,  $\nu_0(r)=\frac{f(r)}{r^d}$, $r\in(0,1)$, where $f(r)$ is non-increasing and non-negative and let  $\beta>0$. If $f(\lambda r)\leq c\lambda^{-\beta}f(r)$, for $\lambda>1$ and $\lambda r\leq 1$  then due to Proposition \ref{muAsymbol} and Theorem \ref{Harnack} the scale invariant Harnack inequality holds.
\end{example}

\begin{example}
Let $A=0$, $\nu_0(r)=\frac{f(r)}{r^d}$, $r\in(0,\infty)$, where $f(r)$ is non-increasing and non-negative  and let $\beta>0$. If $f(\lambda r)\leq c\lambda^{-\beta}f(r)$, for $r>0$ and $\lambda>1$ then the global scale invariant Harnack inequality holds.
For instance this example is applicable for the following processes, ($\alpha,\alpha_1\in(0,2)$):
\begin{itemize}
\item Isotropic $\alpha$-stable process ($f(r)=r^{-\alpha}$), for $d\geq 1$.
\item Relativistic stable process ($f(r)\approx r^{-\alpha}(1+r)^{(\alpha+d-1)/2}e^{-r} $), for $d\geq 1$.
\item Truncated stable process ($f(r)=r^{-\alpha}\textbf{1}_{(0,1)}(r)$).
\item Tempered stable process ($f(r)=r^{-\alpha}e^{-r}$).
\item Isotropic Lamperti stable process ($f(r)=re^{\delta r}(e^s-1)^{-\alpha-1}$, $\delta<\alpha+1$).
\item Layered stable process ($f(r)=r^{-\alpha}\textbf{1}_{(0,1)}(r)+ r^{-\alpha_1}\textbf{1}_{[1,\infty)}(r)$).
\end{itemize}
The scale invariant Harnack inequality for all these examples are known, for instance by \cite{ChenKimKumagai}, but to our best knowledge the global  one only for the first and the last one (see e.g. \cite{ChenKumagai}). Another example  to which our result applies is $f(r)= r^{-2}\log^{-2}(1+r^{-\delta})$, for $\delta<1$.
 Note that $f$ does not satisfy the condition (1.5) in \cite{ChenKimKumagai}, so the scale invariant Harnack inequality can not be concluded from \cite{ChenKimKumagai}.
\end{example}

\begin{example}
Let $\phi$ be a Bernstein function comparable with a function regularly  varying at infinity with index $\alpha$.
If $\alpha \in (0,1]$ the scale invariant Harnack inequality holds for the corresponding  subordinate Brownian motion. This covers for instance  results of \cite{Mimica},
where a particular $\phi(\lambda)=\frac{\lambda}{\log(1+\lambda)}-1$ was considered, for which we even have the global scale invariant Harnack inequality due to Theorem \ref{Harnack}.
\end{example}

\begin{example}
Let $\psi_1$ satisfy WLSC$(\beta_1,0,c_1)$ and $\psi_2$ satisfy WLSC$(\beta_2,0,c_2)$. Then $\psi_1+\psi_2$ satisfies WLSC$(\beta_1\wedge\beta_2,0,c_1\wedge c_2)$. Hence, if $\psi_1$, $\psi_2$ are  L\'{e}vy-Khinchine exponents of  isotropic unimodal L\'{e}vy processes, then the global scale invariant Harnack inequality holds with constant depending only on dimension,  $\beta_1\wedge\beta_2$ and $ c_1\wedge c_2$. In particular  the global scale invariant Harnack inequality holds for a sum of two independent isotropic $\alpha$-stable process with exponents $\psi_1(x)=b_1|x|^{\alpha_1}$ and $\psi_2(x)=b_2|x|^{\alpha_2}$, where $0<\alpha_1\leq \alpha_2\leq 2$. Moreover, the constant in the Harnack inequality depends only on dimension and $\alpha_1$ in this case.
\end{example}

\begin{example}
Let $X_t$ be  an isotropic unimodal L\'{e}vy process with the characteristic exponent $\psi$, independent of a Brownian motion $B_t$, then the scale invariant Harnack inequality holds for $X_t+aB_t$, $a>0$. If additionally $\psi$ satisfies the following
$$\psi(x)\leq C \kappa ^{-\beta}\psi(\kappa x),\qquad |x|\leq 1,\, \kappa<1,$$
for some constants $C$ and $\beta>0$, then the global scale invariant Harnack inequality holds for  $X_t+aB_t$, $a>0$.
\end{example}
Of course for all of the above examples H\"{o}lder continuity for bounded harmonic functions holds as well.

\section{Applications to more general L\'{e}vy Processes}
Let $X_t$ be a general L\'{e}vy process and $d\geq 3$. In this section we relax the assumptions and comment on  validity of the previous results in this new setting.

We set three conditions which to some extent  replace the core assumption of the previous section that the process is isotropic unimodal.
\begin{description}
  \item[(A1)] Assume that $\nu(dx)=\nu(x)dx$ and there exist  constants $C^*_1$, $R>0$ such that $$\nu(x-y)\leq C^*_1r^{-d}\int_{B(x,r)}\nu(y-z)dz, \qquad \mathrm{for any }\,\,r<|x-y|/2\wedge R.$$
	
  \item[(A2)] Assume that $G(dx)=G(x)dx$, $x\neq 0$, and there are  constants $C^*_2$, $R>0$ such that $(C^*_2)^{-1} \tilde{G}(|x|)\leq G(x)\leq C^*_2 \tilde{G}(|x|)$, for $|x|\leq R$ and $\tilde{G}$ is non-increasing.
  \item[(A3)] There exists a constant $C^*_3$ such that $|\Im \psi(x)|\leq C^*_3 \Re \psi(x)$ and  $\psi^*(|x|)\leq C^*_3 \Re\psi(x)$, $x\in\Rd$.
  \end{description}

Notice that under (A3) process is transient ($d\geq3$).

In Remark \ref{hApproxRem} we explain that the claim of Corollary \ref{hApproxLemma} holds if  $|\Im\psi(x)|\leq C^*_3\Re\psi(x)$, $x\in\R^d$. Of course then the comparability constant will  depend on $C^*_3$. This condition is also sufficient to get (\ref{JumpProb1}). The second claim of Corollary \ref{JumpProb} holds if we assume additionally that $\Re\psi$ satisfies WLSC$(\beta,\theta,C)$. If we assume {(A3)} we infer the claim of Lemma \ref{Laplace_potential_LP}. Indeed, under (A3) we have
$$\lambda\mathcal{L}f(\lambda)=\frac{1}{(4\pi)^{d/2}}\int_{\R^d}e^{-|x|^2/4}\frac{\Re\psi(x)}{|\psi(x)|^2}dx\approx \int_{\R^d}e^{-|x|^2/4}\frac{dx}{\psi^*(|x|)}.$$
In the proof of Proposition \ref{PotMeasureLP} we used only Lemma \ref{Laplace_potential_LP} and Lemma \ref{fStarScalling}, hence the conclusion of Proposition \ref{PotMeasureLP} holds under (A3).

In the proof of a counterpart of Proposition \ref{CapApprox} and (\ref{CapGeneral}) we use the following theorem.
\begin{theorem}{(\cite{Hawkes1979}, Theorem 3.3)}\label{HawkesThm}
Let $X_1$ and $X_2$ be L\'{e}vy processes having exponents $\psi_1$ and
$\psi_2$, and capacities $\mathrm{Cap}^{\lambda}_1$ and $\mathrm{Cap}_2^\lambda$ respectively. If $\lambda>0$ and $C > 0$
are such that
$$\Re \left(\frac{1}{\lambda +\psi_1(x)}\right) \leq  C\Re\left(\frac{1}{\lambda +\psi_2(x)}\right)\quad \mathrm{for all }\,\, x,$$
then
$$\mathrm{Cap}^\lambda_2(A)\leq 4C  \mathrm{Cap}^\lambda_1(A)$$
for any analytic set $A$.
\end{theorem}
\begin{lemma}\label{SBM_LP}
For any L\'{e}vy process $X_t$ there exists a subordinate Brownian motion with the characteristic exponent $\phi(|x|^2)$ such that, for $r\geq 0$,
$$\frac{1}{8(1+2d)}\phi(r^2)\leq\psi^*(r)\leq 4\phi(r^2).$$
\end{lemma}
\begin{proof}
Let us define a Bernstein function $\phi$ by the formula
$$\phi(r)=\int_{\R^d}\left(1-e^{-|z|^2r}\right)\nu(dz)+||A||r.$$
Since $\frac{1}{2}\left(1\wedge u\right)\leq 1-e^{-u}\leq 1\wedge u$
$$\frac{1}{2}\int_{\R^d}\left(1\wedge(|z|r)^2\right)\nu(dz)+||A||r^2\leq \phi(r^2)\leq \int_{\R^d}\left(1\wedge(|z|r)^2\right)\nu(dz)+||A||r^2,$$
which completes the proof due to Lemma \ref{PsiStarApprox}.
\end{proof}
The following proposition is a counterpart of Proposition \ref{CapApprox} and (\ref{CapGeneral}).
\begin{proposition}
Let (A3) hold. Then
$$\mathrm{Cap}(\overline{B}_r)\approx r^{d}\psi^*(r^{-1}).$$
Moreover, for any non-empty Borel set $A$ we get
$$\mathrm{Cap}(A)\geq C(d,C^*_3)\psi^*(|A|^{-1/d})|A|.$$
\end{proposition}
\begin{proof}Let $\lambda>0$. Then, by (A3),
$$\frac{1}{(1+(C^*_3)^2)(\lambda+ \psi^*(|x|))}\leq \Re \left(\frac{1}{\lambda +\psi(x)}\right)\leq \frac{C^*_3}{\lambda+ \psi^*(|x|)}.$$
Hence, by Theorem \ref{HawkesThm} and Lemma \ref{SBM_LP} we obtain, for any analytic set $A$,
$$\frac{1}{32(1+8d)C^*_3}\mathrm{Cap}^\lambda_{\phi}(A)\leq \mathrm{Cap}^\lambda(A)\leq 16(1+(C^*_3)^2)  \mathrm{Cap}^\lambda_{\phi}(A),$$
where $ \mathrm{Cap}^\lambda_{\phi}$ denote the capacity of a subordinate Brownian motion $Y_t$ with the characteristic exponent $\phi$ defined in Lemma  \ref{SBM_LP}. Since $\mathrm{Cap}(A)=\lim_{\lambda\to 0^+}\mathrm{Cap}^{\lambda}(A)$ the above inequality holds also for $\lambda=0$.
Finally, we use (\ref{CapApprox}), (\ref{CapGeneral}) for $Y_t$ and again Lemma  \ref{SBM_LP} to get the conclusion.
\end{proof}
To get conclusions of Theorem \ref{PotentialKernelLP}, Corollary \ref{GreenPotentialGreenFunction} and Proposition \ref{GreenLowerLP} it is enough to assume (A2), (A3) and the weak lower scaling condition for $\Re\phi$. Under the same assumptions Proposition \ref{PropKrylovSafonov} holds. We additionally need  to assume (A1) to prove Proposition \ref{PropSuppportOut}. Finally we have the following theorems.
\begin{theorem}
Suppose that (A1)-(A3) hold and $\Re\psi$ satisfies the weak lower scaling condition, then the scale invariant Harnack inequality holds.
\end{theorem}
\begin{theorem}
Suppose that (A1)-(A3) hold and $\Re\psi$ satisfies the weak lower scaling condition, then the conclusion of Theorem \ref{Holder} hold.
\end{theorem}
\begin{remark}
If (A1), (A2) holds for $R=\infty$ and $\Re\psi$ satisfies the global weak lower scaling condition then the global scale invariant Harnack  inequality holds.
\end{remark}

For instance, we can use our results to the sum of two independent isotropic stable processes with drift. More precisely we consider a process  $X_t$ with $\psi(x)=|x|^{\alpha_1}+|x|^{\alpha_2}-i\left<x,\gamma\right>$, where $1<\alpha_1<2$,  $\alpha_2\leq 1$ and $\gamma\in\R^d$. It is easy to see that this process satisfy (A1) and (A3) and the global weak lower scaling condition. Since estimates for the heat kernel of this process are known locally in time (see \cite{Szczypkowski}) and estimates for  the heat kernel of the sum of two independent isotropic stable process (\cite{ChenKumagai}) one can check that potential kernels of these two processes are locally comparable. Hence the assumption (A2) is satisfied. Therefore we infer that the scale invariant Harnack inequality holds for $X_t$.

Note that similar condition as (A1) appeared in \cite{BogdanSztonyk} and \cite{BarlowBassKumagai} and exactly the same in \cite{ChenKimKumagai2} and  \cite{ChenKimKumagai}.
Instead of conditions (A2) and  (A3) the authors of the above mentioned papers assumed some additional conditions for a L\'{e}vy measure.

Now, we discuss the case $d\leq 2$. We say that a function $f:(0,\infty) \to (0,\infty)$ is almost increasing if there exists a constant $c$ such that for any $0<x<y$,  $f(x)\leq c f(y)$. Assume again that  $X_t$ is  an isotropic unimodal L\'{e}vy process. We pose a question if we can apply directly our method without assuming that $X_t$ is a projection as in Corollary \ref{proj}.   First of all let us notice that our approach requires  the process to be transient which holds for any isotropic L\'{e}vy process  for $d\geq 3$. In fact  we  not  only need transience but we also use the property that a function $r\to r^{d-2}\psi^*(1/r)$ is almost increasing for $d\geq 3$. Hence, if instead of $d\geq 3$ we  assume that a function $r\to r^{d-\varepsilon}\psi^*(1/r)$ is almost increasing for some $\varepsilon>0$ we obtain all results from Subsection 3.1 and 3.3, but with all the constants dependent on the process. In particular this assumption implies that the process is transient.

\vspace{10pt}

\noindent
{\large {\textbf{Acknowledgements}}}\\
This research are supported by the Alexander von Humboldt Foundation. This paper was  prepared during my
stay at the Technische Unversit\"{a}t Dresden.  I want to express
my gratitude for the hospitality provided by this university.  I would like to thank Prof. Rene Schilling and Prof. Micha\l{} Ryznar for  many stimulated discussions and the help in  preparation of this paper.

\end{document}